\documentclass[10pt,a4paper]{article}
\usepackage[english]{babel}
\usepackage{amsmath,amssymb,amsfonts,amsthm}
\usepackage{mathrsfs,mathabx}
\usepackage{comment}
\usepackage{color}
\usepackage{appendix}
\usepackage{placeins}
\usepackage{multirow, makecell, array, enumerate}
\usepackage{colortbl}
\usepackage{bm} 
\usepackage{multicol}
\usepackage{hyperref}
\usepackage{dsfont}
\usepackage{graphicx}
\usepackage{subcaption}
\usepackage{algorithmicx}
\usepackage{algorithm}
\newtheorem{proposition}{Proposition}

\newtheorem*{proof*}{Proof}
\definecolor{gray}{rgb}{0.5, 0.5, 0.5}
\definecolor{white}{rgb}{1, 1, 1}
\newtheorem{remark}{Remark}
\algdef{SE}[DOWHILE]{Do}{doWhile}{\algorithmicdo}[1]{\algorithmicwhile\ #1}
\definecolor{verde}{rgb}{0.2, 0.5, 0.3}

\definecolor{magenta}{rgb}{1.0, 0.0, 1.0}

\definecolor{verde}{rgb}{0.2, 0.5, 0.3}

\usepackage{algpseudocode}
\usepackage{tikz}
\usetikzlibrary{positioning, shapes}

\newcommand{\calZ}{\mathcal Z}
\newcommand{\calU}{\mathcal U}

\newcommand{\R}{\mathbb R}

\makeatletter
\let\@fnsymbol\@arabic
\makeatother

\title{Statistical Variational Data Assimilation}

\author{Amina \scshape{Benaceur}\thanks{Institut f\"ur Angewandte und Numerische Mathematik, Karlsruher Institut f\"ur Technologie, Englerstr. 2, 76131 Karlsruhe, Germany} \and Barbara \scshape{Verf\"urth}\thanks{Institut f\"ur Numerische Simulation, Universit\"at Bonn, Friedrich-Hirzebruch-Allee 7, 53115 Bonn}}

\date{}
\begin{document}
	\maketitle
	\begin{abstract}
		This paper is a contribution in the context of variational data assimilation combined with statistical learning.
		The framework of data assimilation traditionally uses data collected at sensor locations in order to bring corrections to a numerical model designed using knowledge of the physical system of interest.
		However, some applications do not have available data at all times, but only during an initial training phase. Hence, we suggest to combine data assimilation with statistical learning methods; namely, deep learning.
		More precisely, for time steps at which data is unavailable, a surrogate deep learning model runs predictions of the `true' data which is then assimilated by the new model.
		In this paper, we also derive \textit{a priori} error estimates on this SVDA approximation. 
		Finally, we assess the method by numerical test cases.
		\newline
		
	\end{abstract}
	
	\section{Introduction}
	
	State estimation is a task in which the quantity of interest is the `true' state $u^{\rm true}$ of a physical system over a space or space-time domain of interest.
	In general, numerical prediction using mathematical models based on the physical knowledge of a system may be deficient due to limitations imposed by available knowledge. 
	Data assimilation (DA) has the goal to overcome these limitations and produce more accurate predictions by incorporating experimental observations in numerical models.
	For this reason, data assimilation methods have been widely explored in the literature.
	
	The goal of these methods is to use \textit{a priori} information to deduce the best mathematical model,
	while using available experimental data to produce the most accurate approximation of a physical system. 
	Many data assimilation methods involve the minimization of a cost function.
	As opposed to statistical techniques such as bayesian data assimilation~\cite{bayes} or Kalman filtering~\cite{kalman}, variational data assimilation relies more heavily on the mathematical model. 
	Widely studied methods are the 3D-VAR~\cite{3dvar_veroy} and the 4D-VAR~\cite{4d-var}.
	Like many variational data assimilation methods, one of the drawbacks of 3D-VAR and 4D-VAR is their computational intrusivity, which means that at any stage, computational procedures need to access the model in order to perform their calculations. Intrusivity is very inconvenient in many contexts, for instance when using industrial high-fidelity black-box solvers. Hence, non-intrusive or partially-intrusive
	options can be valuable. 
	The parametrized background data weak (PBDW) method introduced in~\cite{yano} is non-intrusive. 
	It is a special case of 3D-VAR. It has been studied in many further works, with the presence of noise~\cite{yano_noise,hammond}, and in time-independent contexts~\cite{nki1,nki2}.
	Another drawback of variational data assimilation is that it assumes that data is available at all times, which is not the case in many industrial contexts. For instance, in numerical weather prediction, data is collected using weather balloons to be sent at predefined times. Another common industrial framework is that in which data collection campaigns are conducted during limited periods of time, whereas numerical investigation of the physical phenomena of interest may be carried out during longer periods.
	
	The key idea in this contribution is to use machine learning (ML) in combination with data assimilation in the above described context of partially available real data.
	In the literature, other works bring data assimilation and statistical learning methods together. 
	Equivalences between data assimilation and machine learning are discussed in the review paper~\cite{geer}, which provides a review of existing methods in the context of earth systems.
	In~\cite{ruck}, a DA-ML method is suggested to constrain the output for mass conservation. The chosen data assimilation method uses Ensemble Kalman filtering, and statistical learning is performed \textit{via} a convolutional neural network.
	Moreover,~\cite{brajard} suggests an offline DA-ML method, with iterative application of Ensemble Kalman filtering and convolutional neural networks. 
	The goal is to estimate the state at the current step given the state at the previous one.
	The core idea is to derive a convolutional neural network trained using data assimilation during the offline stage.
	Since data assimilation is only used offline, the resulting online processing is of statistical nature.
	Regarding the use of variational methods for data assimilation,~\cite{farchi} explores two ideas. 
	It compares the performance of DA-ML for resolvent correction and tendency correction using a modified 4D-VAR formulation.
	The offline training uses the DA-ML method introduced in \cite{brajard}.
	An online training of the neural network is also explored in~\cite{farchi}.
	
	In this paper, we try to overcome the data-availability concern by combining data assimilation with statistical prediction.
	The method --- we call it statistical variational data assimilation (SVDA) --- is split into two stages: an offline stage during which we train a deep learning method, and an online stage during which we run our data assimilation model, with real observations replaced by their statistically predicted counterparts.
	In order to circumvent the intrusivity of the classical data assimilation methods, we implement our method using the PBDW approach.
	
	The paper is organized as follows.
	Section~\ref{sec:svda} sets the notation and introduces the main idea of the SVDA. 
	Section~\ref{sec:DA} presents the chosen data assimilation and deep learning methods, respectively, and thereby the details of the SVDA.
	In Section~\ref{sec:analysis}, we present an error analysis of the method.
	Finally, in Section~\ref{sec:num}, the SVDA is illustrated by some numerical results.

	\section{Statistical Variational Data Assimilation}\label{sec:svda}
	In this context, we define a so-called `best-knowledge' model (\texttt{bk}) as the best possible model established by human expertise.
	The \texttt{bk} model is a mathematical model rendering the behavior of the system given all available knowledge about a physical system. 
	More concretely, we will use a PDE or a set of PDEs that best models the physical problem.
	
	In the literature, variational data assimilation methods are data-driven approaches that use observations to improve existing \texttt{bk} models.
	Data is collected, e.g. using sensors, and integrated to the \texttt{bk} model so as to build a data-driven model. 
	Algorithm~\ref{alg:DA} shows the steps of variational data assimilation.
	\begin{algorithm}[htb]
		\caption{Variational data assimilation\label{alg:DA}}
		\begin{algorithmic}[1]
			\vspace{0.2cm}
			\State Formulate a \texttt{bk} model.
			\State Set sensors at user-defined locations.
			\For{\text{every new calculation}}
			\State Collect observations at sensor locations.\label{line:collect}
			\State Solve the data assimilation problem using the collected observations as input.\label{line:solve}
			\EndFor
			\vspace{0.2cm}
		\end{algorithmic}
	\end{algorithm}
	However, ready-to-use data is not always available.
	In fact, a common engineering scenario is that in which measurements/observations are available for a given time interval or for given parametrizations explored during a training phase.
	This training phase is more commonly referred to as a `test campaign' in industrial and engineering scenarios.
	Thus, it is of great interest to develop appropriate and optimal use of the collected observations to understand the behavior of the system.
	In this section, we introduce a new method to overcome the issue of unavailable observations.
	More precisely, we do not deal with cases where some observations have been collected inaccuretely or are missing~\cite{missing_data}.
	We rather deal with situations in which data has been collected for given time windows but is not available at later times.
	Towards this end, we suggest to use prediction-based methods in order to approximate the observations.
	The predicted observations will then be plugged into a data assimilation framework as a surrogate to real-time data.
	In this paper, we use the Parametrized Background Data-Weak approach (PBDW) as a variational data assimilation method.

	As mentioned, we need to perform a statistical learning step in the SVDA.
	The goal of this step is to predict the alternative data needed for the data assimilation problem using chosen input variables.
	Many methods can be used depending on which inputs will be used to train the model.
	As an example, the behavior of linear phenomena can be predicted using linear regression~\cite{hastie_intro}, and nonlinear phenomena can be addressed using polynomial regression, or neural networks~\cite{hastie_elements}.
	In this paper, the phenomena of interest are time-dependent and nonlinear.
	Hence, in the presented framework, the SVDA relies on a type of Neural Networks (NN) called Long Short-Term Memory Recurrent Neural Networks (LSTM-RNN). Yet, the method can be run using any other deep learning or machine learning method~\cite{goodfellow}.
	The choice should be made in light of the intricacies of the problem of interest.

	\subsection{Main Ideas}
	Consider a finite time interval $I=[0,T]$, with $T>0$.
	To discretize in time, we consider an integer $K\geq1$, we define
	$ 0=t^0<\cdots<t^K=T$ as $(K+1)$ distinct
	time nodes over $I$, and we set $ \mathbb K^\mathrm{tr} = \{1,\ldots,K\}$, 
	$ \overline {\mathbb K}^\mathrm{tr} = \{0\}\cup\mathbb K^\mathrm{tr}$ and
	$I^\mathrm{tr} = \{t^k\}_{k\in \overline{\mathbb K}^\mathrm{tr}}$.
	This section aims at deriving a state estimate for a time-dependent solution.
	
	We assume the existence of an initial offline phase $[0,\Delta t]$ of width $\Delta t>0$, and an integer $k_{\rm off}>0$ such that 
	$\Delta t = t^{k_{\rm off}-1}$.
	In other words, $k_{\rm off}$ is the first index at which the observations are unavailable.
	During the offline phase, we use the data available in the time window $[0,\Delta t]$ to predict the evolution of the observables via a deep learning model (LSTM-RNN).
	The offline stage also consists of two additional steps: building a background space $\mathcal Z_N$ of dimension $N$, and building an observable space $\mathcal U_M$ of dimension $M$. 
	These spaces respectively model the solution space for the \texttt{bk} model (e.g. the PDE of interest) and the space used to capture the observations.
	Additional details on the construction of such spaces will be given in Section~\ref{sec:DA} below.
	
	Using the statistical (deep learning) LSTM-RNN model, data for future times (starting at time index $k_{\rm off}$) is predicted for all sensor locations.
	Despite the absence of real data, it now becomes possible to perform a more informed simulation for the next time steps without relying solely on the \texttt{bk} model.
	In fact, the LSTM-RNN provides a surrogate prediction of the system response to circumvent the collection step in line~\ref{line:collect} of Algorithm~\ref{alg:DA}.
	The data assimilation step in line~\ref{line:solve} of Algorithm~\ref{alg:DA} can then be run using a prediction of the true unavailable data observations.
	Hence, the final solution is obtained through a combination of a \texttt{bk} model and statistical learning.
	The online stage of the SVDA method is presented in Algorithm~\ref{alg:on_svda}.
	
	\begin{algorithm}[htb]
		\caption{Online stage of the SVDA\label{alg:on_svda}}
		\begin{algorithmic}[1]
			\vspace{0.2cm}
			\Statex {\scshape{\bfseries \underline {Input :}}} \texttt{bk} model, statistical model, $k_{\rm off}$.
			\vspace{0.2cm}
			\For{$k \in \{k_{\rm off},\ldots,K\}$}
			\State Compute the statistical prediction of observations at time $t^k$.
			\State Solve the SVDA online system at time $t^k$.
			\EndFor
			\vspace{0.2cm}
			\Statex {\scshape{\bfseries \underline  {Output :}}}
			SVDA future trajectory $\{u^{k,\rm svda}_{N,M}\}_{k_{\rm off}\leq k\leq K}$.
			\vspace{0.2cm}
		\end{algorithmic}
	\end{algorithm}

	\section{SVDA and PBDW Problem Formulation}\label{sec:DA}
	We consider a spatial domain (open, bounded, 
	connected subset) $\Omega \subset \mathbb{R}^d$, $d \geq 1$, with a Lipschitz boundary.
	We introduce a Hilbert space $\mathcal U$ composed of functions defined over $\Omega$. The space $\mathcal U$ is endowed with an inner product $\left( \cdot,\cdot \right)$ and we denote by $\|\cdot\|$ the induced norm; $\mathcal U$ consists of functions $\{w:\Omega\rightarrow \mathbb R\ |\ \|w\|<\infty \}$. To fix the ideas, we assume that $H^1_0(\Omega)\subset \mathcal U\subset H^1(\Omega)$, and we denote the dual space of $\mathcal U$ by $\mathcal U'$.
	The Riesz operator $R_\mathcal U:\mathcal U'\rightarrow \mathcal U$ satisfies, for each $\ell\in \mathcal U',$ and for all $v\in \mathcal U$, the equality
	$\left( R_\mathcal U(\ell),v \right)=\ell(v).$
	For any closed subspace $\mathcal Q\subset\mathcal U$, the orthogonal complement of $\mathcal Q$ is defined as 
	$\mathcal Q^\perp:=\{w\in\mathcal U\ |\left( w,v \right)=0,\ \forall v\in\mathcal Q\}.$
	Finally, we introduce a parameter set $\mathcal P\subset \mathbb R^p,\ p\geq 1$, whose elements are
	generically denoted by $\mu\in\mathcal P$.
	
	We recall that we use a `best-knowledge' (\texttt{bk}) mathematical model in the form of a parametrized PDE posed over the domain $\Omega$ (or more generally, over a domain $\Omega^{\rm bk}$ such that $\Omega\subset \Omega^{\rm bk}$).
	Then, we introduce the manifold associated with the solutions of the \texttt{bk} model $\mathcal M^{\rm bk}\subset \mathcal U$.
	In ideal situations, the true solution $u^{\rm true}$ is well approximated by the \texttt{bk} manifold, i.e., the model error 
	\begin{equation}\label{eq:eps_bk}
		\epsilon^{\rm bk}_{\rm mod}(u^{\rm true}):=\underset{z\in\mathcal M^{\rm bk}}{\rm inf}\ \|u^{\rm true}-z\|,
	\end{equation} 
	is very small.
	
	We introduce nested background subspaces $\mathcal Z_1\subset\ldots\subset \mathcal Z_N\subset\ldots\subset\mathcal U$ that are generated to approximate the \texttt{bk} manifold $\mathcal M^{\rm bk}$ to a certain accuracy. These subspaces can be built using various model-order reduction techniques, for instance, the Reduced Basis Method~\cite{rbm3,rbm1,rbm2}. Note that the indices of the subspaces conventionally indicate their dimensions. To measure how well the true solution is approximated by the background space $\mathcal Z_N$, we define the quantity
	$\epsilon^{\rm bk}_N(u^{{\rm true}}):=\inf_{z\in\mathcal Z_N}\|u^{{\rm true}}-z\|$.
	The background space is built so that
	$\epsilon^{\rm bk}_N(u^{{\rm true}})
	\underset{N\rightarrow +\infty}{\rightarrow} \epsilon^{\rm bk}_{\rm mod}(u^{{\rm true}})$.
	Moreover, we introduce the reduction error
	$\epsilon^{\rm bk}_{{\rm red},N} :=\underset{u\in\mathcal M^{\rm bk}}{\rm sup}\underset{z\in\mathcal Z_N}{\rm inf}\|u-z\|$,
	which encodes the loss of accuracy caused by solving the \texttt{bk} model in the $N$-dimensional background space $\mathcal Z_N$.
	For later purposes, we introduce $\Pi_{\mathcal Z_N}(u^{\rm true})$ as the closest point to $u^{\rm true}$ in $\mathcal Z_N$. 
	Note that $\Pi_{\mathcal Z_N}$ is the $\mathcal U$-orthogonal projection onto $\mathcal Z_N$.
	The background space $\mathcal Z_N$ can be interpreted as a prior space that approximates the \texttt{bk} manifold which we hope approximates well the true state $u^{{\rm true}}$.
	As previously alluded to, $u^{{\rm true}}$ rarely lies in $\mathcal M^{\rm bk}$ in realistic engineering study cases.
	In the remainder of this Section, we give a brief review of the PBDW method and present the SVDA setting in the PBDW context.

	\subsection{PBDW formulation in the time-dependent context}
	Since the SVDA builds upon the PBDW, we recap the main ideas of the PBDW formulation for time-dependent problems, following \cite{nki1,nki2}.
	
	Even when one can collect observations, full-knowledge of $u^{\rm true}$ is unrealistic. 
	In many engineering cases, only a limited number of experimental observations of the true state $u^{\rm true}$ is affordable --- interpreted as the application of prescribed observation functionals $\ell^{\rm obs}_m\in\mathcal U'$ for all $m\in\{1,\ldots,M\}$.
	One can consider any observation functional that renders the behavior of some physical sensor.
	
	In the time-dependent setting, we let these observation functionals act on time-averaged snapshots of the true solution.
	We also make the regularity assumption $u^{\rm true}\in L^1(I;\mathcal U)$ on the true state.
	We introduce the time-integration intervals
	\begin{equation}
		\mathcal I^k=[t^k-\delta t^k,t^k+\delta t^k],\quad\forall k\in\mathbb K^{\rm tr},
	\end{equation}
	where $\delta t^k>0$ is a parameter related to the time-precision of sensors.
	Then, for any function $v\in L^1(I;\mathcal U)$, we define the time-averaged snapshots
	\begin{equation}
		v^k(x):=\frac{1}{|\mathcal I^k|}\int_{\mathcal I^k} v(t,x)\ dt 
		\in\mathcal U,\quad\forall k\in\mathbb K^{\rm tr}.
	\end{equation}
	Now we consider
	\begin{equation}\label{eq:ell}
		\ell^{k,\rm obs}_m(u^{\rm true}):=\ell^{\rm obs}_m(u^{k,\rm true}),\quad\forall m\in\{1,\ldots,M\}, \ \forall k\in\mathbb K^{\rm tr}.
	\end{equation}
	For instance, if the sensors act through local uniform time integration, we have
	\begin{equation}\label{eq:td:sensors}
		\ell^{k,\rm obs}_m(u^{\rm true}) =\frac{1}{|\mathcal R_m|}\int_{\mathcal R_m} u^{k,\rm true}(x)\ dx
		=\frac{1}{|\mathcal R_m|}\frac{1}{|\mathcal I_k|} \int_{\mathcal R_m} \int_{\mathcal I_k} u^{\rm true}(t,x)\ dx dt.
	\end{equation}
	
	The most convenient configuration to collect observations in industrial contexts is to measure the quantities at user-defined space-time locations. In actual practice, sensors do not take pointwise measures but localized ones. A sensor collects the data that is enclosed in a small area centered at the sensor location. Hence, equation~\eqref{eq:td:sensors} means that the sensor returns a measurement that is equal to the space-time averaged quantity we are collecting. 
	Moreover, the observation functionals $\ell^{k,\rm obs}_m$ are Riesz representors of any type of physical data available for the user, and not necessarily pointwise or pointwise-like (integrations over small patches) observations.

	Generally, we introduce the time-independent observable space $\mathcal U_M\subset\mathcal U$ such that
	\begin{equation}\label{eq:UM}
		\mathcal U_M={\rm Span}\{q_1,\ldots,q_M\},
	\end{equation}
	where $q_m:=R_\mathcal U (\ell_m^{\rm obs})$ is the Riesz representation of $\ell_m^{\rm obs}\in\mathcal U'$, for all $m\in\{1,\ldots,M\}$, i.e., 
	\begin{equation}\label{eq:obs}
		\ell_m^{\rm obs}(u^{k,\rm true})=(u^{k,\rm true},q_m), \quad\forall m\in\{1,\ldots,M\},\quad\forall k\in\mathbb K^{\rm tr}.
	\end{equation}
	
	Note that, for fixed sensor locations, the computational effort to compute the Riesz representations of the observation functionals is time-independent and is incurred only once so that, for all $ m\in\{1,\ldots,M\}$ and $k\in\mathbb K^{\rm tr}$, the experimental observations of the true state satisfy: 
	\begin{align}\label{eq:td:riesz}
		\begin{split}
			\ell_m^{k,\rm obs}(u^{\rm true})&=\left( u^{k,{\rm true}},q_m \right)= \frac{1}{|\mathcal I_k|}\int_{\mathcal I^k}\ell_m^{\rm obs}(u^{\rm true}(t,\cdot))dt.
		\end{split}
	\end{align}
	Hence, for all $q\in\mathcal U_M$ such that,
	\begin{equation}
		q=\sum_{m=1}^M \alpha_m q_m,
	\end{equation}
	the inner product $\left( u^{k,{\rm true}},q \right)$ is deduced from the experimental observations as follows:
	\begin{equation}
		\begin{split}
			\left( u^{k,{\rm true}},q \right)
			&=\frac{1}{|\mathcal I_k|}\int_{\mathcal I^k}\sum_{m=1}^M\alpha_m\left( u^{{\rm true}}(t,\cdot),q_m \right)dt \\
			&=\frac{1}{|\mathcal I_k|}\sum_{m=1}^M\alpha_m\int_{\mathcal I^k}\ell_m^{\rm obs}(u^{{\rm true}}(t,\cdot))dt.
		\end{split}
	\end{equation}

	Henceforth, we make the crucial assumption that
	\begin{equation}\label{eq:ass}
		\mathcal Z_N\cap \mathcal U_M^\perp=\{0\},
	\end{equation} 
	which is also equivalent to
	\begin{equation}
		\beta_{N,M}:=\underset{w\in\mathcal Z_N}{\inf}\
		\underset{v\in\mathcal U_M}{\sup}
		\frac{(w,v)}{\|w\|\ \|v\|}\in(0,1],
	\end{equation}
	where $\beta_{N,M}$ is the so-called stability constant (the reader can refer to~\cite{myphd} for a proof).
	Assumption~\ref{eq:ass} can be viewed as a requirement to have enough sensors (note that $\mathcal Z_N\cap\mathcal U^\perp = \{0\}$).
	Under this assumption, the limited-observations PBDW statement reads: for each $k\in\mathbb K^{\rm tr}$, find $(u_{N,M}^{k,*}, z_{N,M}^{k,*}, \eta_{N,M}^{k,*})\in \mathcal U\times\mathcal Z_N\times \mathcal U_M$ such that
	\begin{subequations}\label{eq:td:limited_obs}
		\begin{align}(u_{N,M}^{k,*}, z_{N,M}^{k,*}&, \eta_{N,M}^{k,*}) = \underset{\eta_{N,M}\in\mathcal U_M}{\underset{z_{N,M}\in\mathcal Z_{N}}{\underset{u_{N,M}\in\mathcal U}{\rm arginf}}}\ 
			\|\eta_{N,M}\|^2, \label{eq:limited_obs1} \\ 
			\text{subject to }
			& (u_{N,M},v)=(\eta_{N,M},v)+(z_{N,M},v),\quad \quad\forall v\in \mathcal U, \\
			& (u_{N,M},\phi)=(u^{k,{\rm true}},\phi),\ \ \qquad\qquad\qquad\forall \phi\in \mathcal U_M.
		\end{align}
	\end{subequations}
	The limited-observations saddle-point problem associated with~\eqref{eq:td:limited_obs} reads: for each $k\in\mathbb K^{\rm tr}$, find $(z_{N,M}^{k,*},\eta_{N,M}^{k,*})\in\mathcal Z_{N}\times\mathcal U_M$ such that
	\begin{subequations}\label{eq:td:limited_euler}
		\begin{align}
			(\eta_{N,M}^{k,*},q)+(z_{N,M}^{k,*},q)&=(u^{k,{\rm true}},q),\quad \forall q\in \mathcal U_M, \label{eq:td:limited_euler1}\\
			(\eta_{N,M}^{k,*},p)&=0,\qquad\quad\quad\ \ \forall p\in \mathcal Z_{N},
		\end{align}
	\end{subequations}
	Hence, the limited-observations state estimate is
	\begin{equation}\label{eq:assemble}
		u_{N,M}^{k,*} = z_{N,M}^{k,*}+\eta_{N,M}^{k,*},\quad\forall k\in\mathbb K^{\rm tr}.
	\end{equation}
	In algebraic form, the limited-observations PBDW statement reads: for each $k\in\mathbb K^{\rm tr}$, find $(\bm z^{k,*}_{N,M},\bm\eta^{k,*}_{N,M})\in \mathbb R^N\times\mathbb R^M$ such that
	\begin{equation}\label{eq:svda:algeb_form0}
		\left(
		\begin{matrix}
			\mathbf A & \mathbf B \\
			\mathbf B^T & \mathbf 0
		\end{matrix}
		\right)
		\left(
		\begin{matrix}
			\bm\eta^{k,*}_{N,M} \\
			\bm z^{k,*}_{N,M}
		\end{matrix}
		\right)
		=
		\left(
		\begin{matrix}
			\bm\ell^{k,{\rm obs}}_M \\
			\mathbf 0
		\end{matrix}
		\right),
	\end{equation}
	with the matrices
	\begin{equation}\label{matricesAB}
		\mathbf A = \Big((q_{m'},q_{m})\Big)_{1\leq m,m'\leq M}\in\mathbb R^{M\times M},\
		\mathbf B = \Big((\zeta_n,q_{m})\Big)_{1\leq m\leq M,1\leq n\leq N}\in\mathbb R^{M\times N},
	\end{equation}
	where $\mathcal Z_N=\operatorname{span}\{\zeta_1,\ldots, \zeta_N\}$,
	and the vector of observations
	\begin{equation}
		\bm\ell^{k,{\rm obs}}_M = \big(\ell^{{\rm obs}}_m(u^{k,\rm true})\big)_{1\leq m\leq M} \in\mathbb R^{M}.
	\end{equation}
	Note that the matrices $\mathbf A$ and $\mathbf B$ are time-independent; only the right-hand side in~\eqref{eq:svda:algeb_form0} depends on $k$.
	
	\begin{remark}[3D-VAR]
		We highlight that the PBDW is a special case of 3D-VAR~\cite{lorenc} variational data assimilation.
		We introduce the bilinear forms $a:\calU\times\calU\rightarrow\R$ and $b:\calU\times\calU\rightarrow\R$, and the linear form $f:\calU\rightarrow\R$.
		In a noise-free context, the 3D-VAR method consists in solving	
		\begin{subequations}\label{3dvar}
			\begin{align}
				(z^{k,*}, \eta^{k,*}) = \underset{\eta\in \mathcal U_M}{\underset{z\in \mathcal U}{\rm arginf}}\ 
				\frac 1 2&\|\eta\|^2 + \frac \lambda 2\|\Pi_{\calU_M}u^{k,\rm true} - \Pi_{\calU_M} (z+\eta)\|^2 , \\
				\nonumber \\ 
				{\rm subject\ to\quad }
				a(z,v)&=f(v)+b(\eta,v),\quad \forall v\in \mathcal U,
			\end{align}
		\end{subequations}
		which is equivalent to solving
		\begin{subequations}\label{3dvar2}
			\begin{align}
				(z^{k,*}, \eta^{k,*}) = \underset{\eta\in\ \mathcal U}{\underset{z\in\ \mathcal Z}{\underset{u\in\ \mathcal U}{\rm arginf}}}\
				&\|\eta\|^2 , \\
				\nonumber \\ 
				{\rm subject\ to\qquad }
				a(z,v)&=f(v)+b(\eta,v),\quad \forall v\in \mathcal U,\label{eq:edp}\\
				(z+\eta,\phi)&=(u^{k,\rm true},\phi),\qquad\quad \forall v\in \mathcal U_M.
			\end{align}
		\end{subequations}
		If $b\equiv 0$, the constraint~\eqref{eq:edp} can be replaced by a requirement that $z$ belongs to the manifold of solutions
		$\calZ = \{u\in \calU\ |\ a(u,v) = f(v) \}$. 
		In this case, the 3D-VAR problem reads:
		\begin{subequations}\label{3dvar3}
			\begin{align}
				(u^*, z^*, \eta^*) = \underset{\eta\in\ \mathcal U}{\underset{z\in\ \mathcal Z}{\underset{u\in\ \mathcal U}{\rm arginf}}}\ 
				&\|\eta\|^2 , \\
				\nonumber \\ 
				{\rm subject\ to\quad }
				(u,v)&=(z,v)+(\eta,v),\qquad \forall v\in \mathcal U,\\
				(u,\phi)&=(u^{\rm true},\phi),\qquad\ \forall v\in \mathcal U_M,
			\end{align}
		\end{subequations}
		which is the same idea as the PBDW formulation.
	\end{remark}
	
	\subsection{SVDA-PBDW formulation}
	Let us now present the SVDA formulation of the PBDW problem.
	We recall that, in the SVDA context, the observations are not available at the time steps of interest.
	Hence, one cannot evaluate the observations $\ell_m^{k,\rm obs}(u^{\rm true})$ in~\eqref{eq:ell}, which will be replaced by their LSTM-RNN prediction.
	On the one hand, the true solution $u^{\rm true}$ is a continuous solution field in $\calU$.
	On the other hand, the LSTM-RNN is a deep learning method that predicts local values of the true solution field at given measurement locations.
	Hence, we will introduce the field $u^{\rm DL} \in\calU_M$ such that
	\begin{equation}
		u^{\rm DL} = \sum_{m=1}^M u^{\rm DL}_m q_m,
	\end{equation}
	where $u^{\rm DL}_m$ is the LSTM-RNN prediction of the $m$-th observation.
	It then holds that $\ell_m^{k,\rm obs}(u^{\rm DL})\approx \ell_m^{k,\rm obs}(u^{\rm true})$. 
	Using the RNN-LSTM prediction $u^{k,\rm DL}$, the problem reads: for each $k\in\mathbb K^{\rm tr}$, find $(z_{N,M}^{k,\rm svda},\eta_{N,M}^{k,\rm svda})\in\mathcal Z_{N}\times\mathcal U_M$ such that
	\begin{subequations}\label{eq:svda:limited_euler}
		\begin{align}
			(\eta_{N,M}^{k,\rm svda},q)+(z_{N,M}^{k,\rm svda},q)&=(u^{k,{\rm DL}},q),\quad \forall q\in \mathcal U_M, \label{eq:svda:limited_euler1}\\
			(\eta_{N,M}^{k,\rm svda},p)&=0,\qquad\quad\quad\ \ \forall p\in \mathcal Z_{N},
		\end{align}
	\end{subequations}
	Hence, the SVDA state estimate is
	\begin{equation}\label{eq:assemble_svda}
		u_{N,M}^{k,\rm svda} = z_{N,M}^{k,\rm svda}+\eta_{N,M}^{k,\rm svda},\quad\forall k\in\mathbb K^{\rm tr}.
	\end{equation}
	In algebraic form, the SVDA-PBDW statement reads: for each $k\in\mathbb K^{\rm tr}$, find $(\bm z^{k,\rm svda}_{N,M},\bm\eta^{k,\rm svda}_{N,M})\in \mathbb R^N\times\mathbb R^M$ such that
	\begin{equation}\label{eq:svda:algeb_form}
		\left(
		\begin{matrix}
			\mathbf A & \mathbf B \\
			\mathbf B^T & \mathbf 0
		\end{matrix}
		\right)
		\left(
		\begin{matrix}
			\bm\eta^{k,\rm svda}_{N,M} \\
			\bm z^{k,\rm svda}_{N,M}
		\end{matrix}
		\right)
		=
		\left(
		\begin{matrix}
			\bm\ell^{k,{\rm DL}}_M \\
			\mathbf 0
		\end{matrix}
		\right),
	\end{equation}
	with the matrices $\mathbf A, \mathbf B$ as in the original PBDW formulation above (cf.~\eqref{matricesAB})
	and the vector of LSTM-RNN predictions of the observations
	\begin{equation}\label{eq:ldl}
		\bm\ell^{k,{\rm DL}}_M = \big(\ell^{{\rm obs}}_m(u^{k,\rm DL})\big)_{1\leq m\leq M} \in\mathbb R^{M}.
	\end{equation}

	\subsection{SVDA in practice}
	SVDA can be used in different contexts.
	This paper mainly focuses on `future predictions', i.e., on generating machine learning observations from a \texttt{bk} model and given observations collected at previous times. The hope is to thereby improve the `future' simulations where real observations are not available.
	Another configuration for the use of the SVDA will be considered in the numerical test cases; namely SVDA in parametric contexts. In that case, the underlying model is a parameterized PDE and observations are available for a certain set of parameter(s). The SVDA can then be used for prediction-based data assimilation for other parameters in the sense that the machine learning model generates a `prediction' for the behavior of the model for the new parameter without real observations, and these machine learning predictions are used as surrogate data in the data assimilation procedure.
	
	Independent of the application case, the SVDA can algorithmically be divided into an offline and an online phase, cf.~Section~\ref{sec:svda}.
	During the \emph{offline stage} (cf.~Algorithm~\ref{alg:off_svda} below), one precomputes the functions $(\zeta_n)_{1\leq n \leq N}$ and the Riesz representers $(q_m)_{1\leq m \leq M}$ leading to the matrices $\mathbf A\in \mathbb R^{M\times M}$ and 
	$\mathbf B\in \mathbb R^{M\times N}$ as in \eqref{matricesAB} once and for all.
	One also runs the deep learning module on the training data in order to produce an LSTM-RNN prediction function, see below for a detailed description.
	
	During the \emph{online stage} (cf.~Algorithm~\ref{alg:on_svda}), one applies the LSTM-RNN prediction function to the input data in order to generate the vector of predicted observations $\bm \ell^{k,{\rm DL}}_M$ in~\eqref{eq:ldl}, see also below.
	Then, we apply the SVDA-PBDW formulation from the previous section, i.e., we solve the linear $(N+M)$-dimensional problem~\eqref{eq:svda:algeb_form} to retrieve $\bm z^{k,\rm svda}_{N,M}$ and $\bm\eta^{k,\rm svda}_{N,M}$.
	Finally, for each $k\in\{k_{\rm off},\ldots,K\}$, the SVDA state estimate $\bm u^{k,\rm svda}_{N,M}$ is deduced as follows:
	\begin{equation}\label{eq:ass_st_est}
		\bm u^{k,\rm svda}_{N,M}=
		\mathbf Z_N \bm z^{k,\rm svda}_{N,M}+
		\mathbf U_M \bm \eta^{k,\rm svda}_{N,M},
	\end{equation}
	where $\mathbf Z_N$ and $\mathbf U_M$ are the algebraic counterparts of $\calZ_N$ and $\calU_M$ respectively.
	
	\begin{algorithm}[htb]
		\caption{Offline stage of the SVDA\label{alg:off_svda}}
		\begin{algorithmic}[1]
			\vspace{0.2cm}
			\Statex {\scshape{\bfseries \underline {Input :}}} $\{q_1,\ldots,q_M\}$: a set of Riesz representations of the observations.
			\vspace{0.2cm}
			\State Compute $\mathcal Z_N={\rm span}\{\zeta_1,\ldots,,\zeta_N\}$.
			\State Set $\mathcal U_M:={\rm Span}\{q_1,\ldots,q_M\}$.
			\State Compute the matrices $\mathbf A$ and $\mathbf B$ using $\mathcal Z_N$ and $\mathcal U_M$.
			\State Compute the LSTM-RNN prediction function $\bm\ell_M^{\rm DL}$.
			\vspace{0.2cm}
			\Statex {\scshape{\bfseries \underline  {Output :}}}
			$\mathcal Z_N$, $\mathcal U_M$, $\mathbf A$ and $\mathbf B$, $\bm \ell_M^{\rm DL}$.
			\vspace{0.2cm}
		\end{algorithmic}
	\end{algorithm}
	
	\paragraph{Statistical training} We now describe the set-up of our LSTM-RNN model for the SVDA. Recall that the specific choice of the machine learning model can be exchanged in the SVDA formulation and should be adapted to the problem context. We therefore focus our description on the concepts used in the numerical experiments below. 
	
	For completeness, we briefly describe how a single LSTM cell works.
	As LSTMs are recurrent networks, there is a recurrent state denoted by $\mathbf{h}$. Additionally, an LSTM unit also includes a cell state $\mathbf{c}$. At unit $j$, $\mathbf{c}_j$ and $\mathbf{h}_j$ are computed from the previous steps $\mathbf{c}_{j-1}$ and $\mathbf{h}_{j-1}$ as well as from the input value $\mathbf{x}_j$, which reflects the input variable at time instance $t_j$. $\mathbf{h}_{j-1}$ and $\mathbf{x}_j$ are concatenated to $\hat{\mathbf{x}}_j$, from which the values between $0$ and $1$ of the \emph{forget, update and output gates} $\mathbf{f}_j, \mathbf{u}_j, \mathbf{o}_j$ are computed via
	\begin{align*}
		\mathbf{f}_j=\sigma(\mathbf{W}_f\hat{\mathbf{x}}_j+\mathbf{b}_f),\quad \mathbf{u}_j=\sigma(\mathbf{W}_u\hat{\mathbf{x}}_j+\mathbf{b}_u),\quad \mathbf{o}_j=\sigma(\mathbf{W}_o\hat{\mathbf{x}}_j+\mathbf{b}_o).
	\end{align*}
	Here, $\mathbf{W}_f, \mathbf{W}_u, \mathbf{W}_o$ and $\mathbf{b}_f, \mathbf{b}_u, \mathbf{b}_o$ are the trainable weights and biases, respectively, and $\sigma$ denotes the activation function. 
	The gate mechanism is the cornerstone of every LSTM unit and determines how the recurrent and cell state are updated. First a new candidate cell state $\tilde {\mathbf{c}}_j$ is computed as \[\tilde{\mathbf{c}}_j=\tanh(\mathbf{W}_c \hat{\mathbf{x}}_j+\mathbf{b}_c),\]
	where $\mathbf{W}_c$ and $\mathbf{b}_c$ are yet another trainable weight and bias, respectively. Then, the updated quantities are calculated via
	\[\mathbf{c}_j= \mathbf{f}_j\odot\mathbf{c}_{j-1}+\mathbf{u}_j\odot \tilde{ \mathbf{c}}_j, \qquad \mathbf{h}_j=\mathbf{o}_j\odot\tanh \mathbf{c}_j,\]
	where $\odot$ denotes the component-wise product. While $\mathbf{c}$ is an internal variable, the recurrent state $\mathbf{h}$ also serves as output, especially at the final LSTM unit.
	We refer the reader to~\cite{lstm} as well as the original articles \cite{lstm1,lstm2} for more details on LSTMs.
	
	For our LSTM-RNN model, we use $lb$ LSTM units with the so-called lookback variable $lb\leq k_{\mathrm{off}}-1$. In other words, this variable determines how many previous time steps are used for the prediction of the observables at the current time. 
	Consequently, the input variable has the shape $\mathbb{R}^{lb\times M}$. 
	Besides the LSTM units, our neural network consists of dense layer(s), see Figure \ref{fig:nn-architecture} for a sketch.
	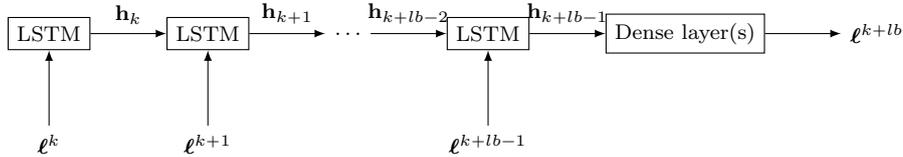
\begin{figure}
		\begin{tikzpicture}[font=\footnotesize]
			\node[draw] (LSTM1) at (0,0) {LSTM};
			\node[draw, right =of LSTM1] (LSTM2) {LSTM};
			\node[right=of LSTM2] (dots) {$\ldots$};
			\node[draw, right=of dots] (LSTM3) {LSTM};
			\node[draw, right=of LSTM3, align=left] (dense1) {Dense layer(s)};
			
			\node[below=of LSTM1](lk) {$\boldsymbol{\ell}^k$};
			\node[below=of LSTM2] (lk1) {$\boldsymbol{\ell}^{k+1}$};
			\node[below=of LSTM3] (lklb1) {$\boldsymbol{\ell}^{k+lb-1}$};
			\node[right=of dense1] (lklb) {$\boldsymbol{\ell}^{k+lb}$};
			
			\draw[-latex] (lk) edge (LSTM1);
			\draw[-latex] (lk1) edge (LSTM2);
			\draw[-latex] (lklb1) edge (LSTM3);
			\draw[-latex] (dense1) edge (lklb);
			
			\draw[-latex] (LSTM1) edge node[pos=0.5, above]{$\mathbf{h}_k$} (LSTM2);
			\draw[-latex] (LSTM2) edge node[pos=0.5, above]{$\mathbf{h}_{k+1}$} (dots);
			\draw[-latex] (dots) edge node[pos=0.5, above]{$\mathbf{h}_{k+lb-2}$} (LSTM3);
			\draw[-latex] (LSTM3) edge node[pos=0.5, above]{$\mathbf{h}_{k+lb-1}$} (dense1);
		\end{tikzpicture}
		\caption{Sketch of the LSTM-RNN setting in the numerical experiments. For better visibility, only the propagation of the recurrent state is depicted.}
		\label{fig:nn-architecture}
	\end{figure}
	The described neural network construction is used to emulate the map from $(\boldsymbol{\ell}^{k}, \ldots \boldsymbol{\ell}^{k+lb-1})$ to $\boldsymbol{\ell}^{k+lb}$ for each training time step $t^k$.
	For fixed $k^{\mathrm{off}}$, $lb$ also influences the number of available input/output training data pairs. To better illustrate this, let us consider two extreme cases. On the one hand, for $lb=1$, we have $k_{\mathrm{off}}-1$ training data pairs. On the other hand, for $lb=k_{\mathrm{off}}-1$, we have only a single training data pair.
	
	Once the network is trained, we can cheaply get predictions $\mathbf\ell^{k,DL}$ for any time step $t_k$ by providing observations at $lb$ previous time steps.
	We, hence, have to sequentially run the neural network several times to get all surrogate observations in $[\Delta t, T]$. Since the evaluation of a trained neural network is rather cheap, this sequential run comes at extremely low cost. 
	The computation of $\mathbf\ell^{k,DL}$ for times $t^k$ in $[\Delta t, T]$ can also be viewed as a machine learning based time stepping.

	\section{Error estimation}\label{sec:analysis}
	In this section, we establish an \textit{a priori} error analysis of the SVDA approximation. 
	We first recall an important result proved in~\cite{yano}.
	\begin{proposition}\label{prop:yano}
		At each time step $k\in\mathbb K^\mathrm{tr}$, the PBDW error estimation satisfies
		\begin{align}\label{eq:yano}
			\|u^{k,\rm true}-u^{k,*}_{N,M}\|
			&\leq
			\left(1+\frac{1}{\beta_{N,M}}\right) \underset{q\in\calU_M\cap\calZ_N^\perp}{\rm inf} \|\Pi_{\calZ_N}u^{k,\rm true}-q\|,
		\end{align}
		where $\beta_{N,M}$ is the stability constant defined as
		\begin{equation*}
			\beta_{N,M} := \underset{z\in\calZ_N}{\rm inf}\underset{q\in\calU_M}{\rm sup}
			\frac{(z,q)}{\|z\|\|q\|} \in (0,1] .
		\end{equation*}
	\end{proposition}
	\begin{proof}
		See~\cite{yano}, Proposition 2.
	\end{proof}
	Let us now estimate the SVDA approximation error.
	\begin{proposition}
		In the present context, the following upper bounds on the SVDA approximation hold true
		\begin{align} \label{eq:diff_star_svda0}
			\|u^{k,*}_{N,M}-u^{k,\rm svda}_{N,M}\|
			\leq
			\left(1+\frac{2}{\beta_{N,M}}\right) 
			\|\Pi_{\calU_M} u^{k,\rm true} - u^{k,\rm DL}\|.
		\end{align}
		\begin{align}\label{eq:diff_star_svda}
			\begin{alignedat}{2}
				\|u^{k,\rm true}-u^{k,\rm svda}_{N,M}\|
				&\leq
				\left(1+\frac{1}{\beta_{N,M}}\right) \underset{q\in\calU_M\cap\calZ_N^\perp}{\rm inf} \|\Pi_{\calZ_N}u^{k,\rm true}-q\| \\ &+\left(1+\frac{2}{\beta_{N,M}}\right) 
				\|\Pi_{\calU_M} u^{k,\rm true} - u^{k,\rm DL}\|.
			\end{alignedat}
		\end{align}
	\end{proposition}
	\begin{proof}
		\begin{enumerate}
			\item By subtracting~\eqref{eq:svda:limited_euler1} from~\eqref{eq:td:limited_euler1}, we obtain
			\begin{equation*}\label{eq:pbdw_diff}
				(\eta^{k,*}_{N,M} - \eta^{k,\rm svda}_{N,M},q) + (z^{k,*}_{N,M} - z^{k,\rm svda}_{N,M},q) = (u^{k,\rm true} - u^{k,\rm DL},q),\qquad \forall q\in\calU_M.
			\end{equation*}
			Let us choose the test function $q=\eta^{k,*}_{N,M} - \eta^{k,\rm svda}_{N,M}\in\calU_M$ as a test function.
			Using the fact that $\eta^{k,*}_{N,M} - \eta^{k,\rm svda}_{N,M}\in\calZ_N^\perp$ along with the Cauchy-Schwarz inequality, we get
			\begin{align*}
				\|\eta^{k,*}_{N,M} - \eta^{k,\rm svda}_{N,M}\|^2 
				&= \left(\Pi_{\calU_M}(u^{k,\rm true} - u^{k,\rm DL}), \eta^{k,*}_{N,M} - \eta^{k,\rm svda}_{N,M}\right)\\
				&\leq \|\Pi_{\calU_M}(u^{k,\rm true} - u^{k,\rm DL})\|\
				\|\eta^{k,*}_{N,M} - \eta^{k,\rm svda}_{N,M}\|.
			\end{align*}
			Hence,
			\begin{equation}\label{eq:majoration_eta}
				\|\eta^{k,*}_{N,M}-\eta^{k,\rm svda}_{N,M}\|
				\leq
				\|\Pi_{\calU_M} (u^{k,\rm true} - u^{k,\rm DL})\|.
			\end{equation}
			Moreover, $(z^{k,*}_{N,M} - z^{k,\rm svda}_{N,M})\in\calZ_N$ leads to
			\begin{align*}
				\begin{alignedat}{2}
					\beta_{N,M} \|z^{k,*}_{N,M} - z^{k,\rm svda}_{N,M}\| 
					&\leq
					\underset{q\in\calU_M}{\rm sup}\frac{(z^{k,*}_{N,M} - z^{k,\rm svda}_{N,M},q)}{\|q\|}\\
					&=
					\underset{q\in\calU_M}{\rm sup}\frac{(u^{k,\rm true} - u^{k,\rm DL} - (\eta^{k,*}_{N,M} - \eta^{k,\rm svda}_{N,M}),q)}{\|q\|}\\
					&\leq 
					\|\Pi_{\calU_M}(u^{k,\rm true} - u^{k,\rm DL})\|
					+\|\eta^{k,*}_{N,M} - \eta^{k,\rm svda}_{N,M}\|\\
					&\leq 
					2\|\Pi_{\calU_M}u^{k,\rm true} - u^{k,\rm DL}\|,
				\end{alignedat}
			\end{align*}
			where the last inequality follows from~\eqref{eq:majoration_eta} and the fact that $u^{k,\rm DL}\in \calU_M$.
			Hence,
			\begin{equation}\label{eq:majoration_z}
				\|z^{k,*}_{N,M} - z^{k,\rm svda}_{N,M}\| 
				\leq
				\frac{2}{\beta_{N,M}} \|\Pi_{\calU_M} u^{k,\rm true} - u^{k,\rm DL}\|.
			\end{equation}
			Combining~\eqref{eq:majoration_eta} and~\eqref{eq:majoration_z}, we obtain
			\begin{equation}\label{eq:majoration_DL}
				\|u^{k,*}_{N,M} - u^{k,\rm svda}_{N,M}\| 
				\leq
				\left(1+\frac{2}{\beta_{N,M}}\right) \|\Pi_{\calU_M} u^{k,\rm true} - u^{k,\rm DL}\|.
			\end{equation}
			The spirit of this proof is similar to that of noisy observations~\cite{yano_noise}.
			\item Using the triangle inequality, we have 
			\begin{align}\label{eq:tri}
				\begin{alignedat}{2}
					\|u^{k,\rm true}-u^{k,\rm svda}_{N,M}\| &\leq \|u^{k,\rm true}-u^{k,*}_{N,M}\| + \|u^{k,*}_{N,M}-u^{k,\rm svda}_{N,M}\|.
				\end{alignedat}
			\end{align}
			It follows from~\eqref{eq:tri}, ~\eqref{eq:yano}, and~\eqref{eq:majoration_DL} that
			\begin{align}
				\|u^{k,\rm true}-u^{k,\rm svda}_{N,M}\|
				&\leq
				\left(1+\frac{1}{\beta_{N,M}}\right) \underset{q\in\calU_M\cap\calZ_N^\perp}{\rm inf} \|\Pi_{\calZ_N}u^{k,\rm true}-q\| \\ &+\left(1+\frac{2}{\beta_{N,M}}\right) 
				\|\Pi_{\calU_M} u^{k,\rm true} - u^{k,\rm DL}\|.
			\end{align}
		\end{enumerate}
	\end{proof}
	The result in~\eqref{eq:diff_star_svda} shows that the quality of the SVDA approximation has two contributions. 
	The first contribution depends on the quality of the PBDW spaces. The better the quality of the background space $\calZ_N$ and the observable space $\calU_M$, the smaller the error $\|u^{k,\rm true}-u^{k,\rm svda}_{N,M}\|$.
	The second contribution is related to the quality of the statistical prediction; in our case, that of the LSTM-RNN.
	The more accurate $u^{k,\rm DL}$, the smaller the error $\|u^{k,\rm true}-u^{k,\rm svda}_{N,M}\|$.
	In concrete applications, both error contributions may be estimated further or assessed with an indicator to yield an overall bound for the SVDA quality.

	\section{Numerical Results}\label{sec:num}
	In this section, we implement the above developments.
	The goal is to illustrate the computational performance of the SVDA method.
	We resort to the following numerical strategy:
	\begin{enumerate}
		\item Synthesize two different models out of the same PDE. Towards that end, an efficient strategy is to change a (physical) parameter. The first model will be considered as the `true' model and the second will be the `best-knowledge' model.
		\item Using the `true' model, create training data for the initial time interval $[0,\Delta t]$.
		\item Train the statistical model using LSTM-RNN.
		\item Run the online SVDA.
	\end{enumerate}
	We consider a two-dimensional setting based on the plate illustrated in the left panel of Figure~\ref{fig:plate} with
	$\Omega=(-2,2)^2\subset \mathbb R^2$.
	\begin{figure}[htb]
		\includegraphics[scale=0.43]{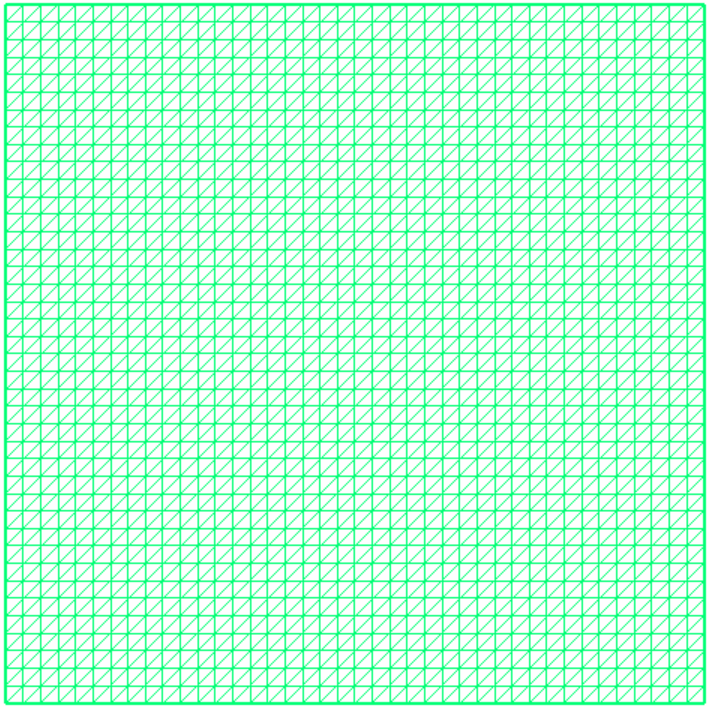} 
		\qquad
		\includegraphics[scale=0.43]{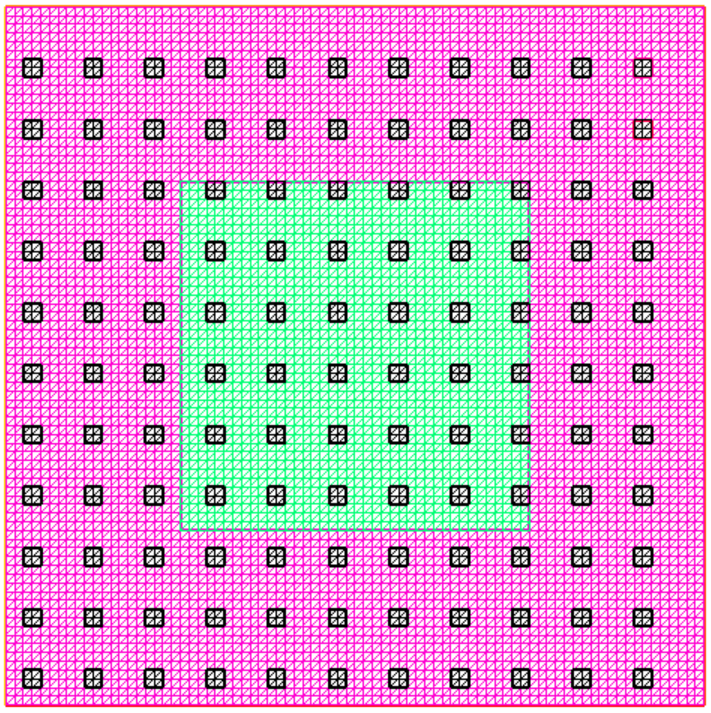} 
		\centering 
		\caption{Computational domain and mesh with $\mathcal N = 6561$.
			The little black squares are observation subsets $\{\mathcal R_m\}_{m=1}^{121}$.
			Left: Mono-material plate.
			Right: Bi-material plate.}
		\label{fig:plate}
	\end{figure}
	We use a finite element (FE)~\cite{ernGuer} subspace $\mathcal U^\mathcal N\subset \mathcal U= H^1(\Omega)$, where $H^1(\Omega)$ is the linear space of square integrable and differentiable functions defined on $\Omega$.
	The subspace $\mathcal U$ consists of continuous, piecewise affine functions in order to generate FE solutions. 
	The FE subspace $\mathcal U^\mathcal N$ is based on a mesh that contains $\mathcal N = 6561$ nodes.
	The experimental data is generated synthetically and the observation subsets $\{\mathcal R_m\}_{1\leq m\leq M}$  are uniformly selected over the plate as illustrated in the right panel of Figure~\ref{fig:plate}.
	Regarding the implementation, the FE computations use the software \texttt{FreeFem++}~\cite{freefem}, the SVDA algorithm has been developed in \texttt{Python}. The deep learning subroutines of the SVDA use the \texttt{Python} library \texttt{Tensorflow.keras}~\cite{keras}.
	
	\subsection{Physical Model Problem}
	We apply the above methodology to the following parabolic PDE: For many values of the parameter $\mu\in\mathcal P$, find 
	$u(\mu):I\times \Omega\rightarrow \mathbb R $ such that
	\begin{align}\label{HTE}
		\left\{
		\begin{alignedat}{2}
			\frac{\partial u(\mu)}{\partial t}
			-\nabla \cdot \left( D(\mu) \nabla u(\mu) \right)& 
			= 0,\qquad &&\text{in}\ I\times\Omega , \\
			u(\mu)(t=0,\cdot)&= u_0,\quad &&\text{in }\Omega,\\
			\text{Boundary cond}&\text{itions},\quad &&\text{on }I\times\partial\Omega,
		\end{alignedat}
		\right.
	\end{align}
	where $u_0=293.15$K ($20^o$C).
	We will supplement~\eqref{HTE} with Stefan--Boltzmann boundary conditions on $\partial\Omega$, i.e.,
	\begin{equation}\label{eq:nonlin_bcs}
		-D(\mu)\frac{\partial u}{\partial n} = \sigma\varepsilon(u^4-u_r^4),\qquad\text{on}\  I\times\partial\Omega,
	\end{equation}
	with an enclosure temperature $u_r = 303.15$K ($30^o$C), the Stefan--Boltzmann constant $\sigma=5.67\times10^{-8}$W.m$^{-2}$.K$^{-4}$, and an emissivity $\varepsilon=3.10^{-3}$. The Stefan--Boltzmann boundary condition is nonlinear.
	Hence, the resulting problem~\eqref{HTE}$-$\eqref{eq:nonlin_bcs} is nonlinear.
	In what follows, the background spaces $\mathcal Z_N$ will be generated by solving the nonlinear PDE~\eqref{HTE}$-$\eqref{eq:nonlin_bcs} with a uniform diffusivity function $D(\mu)$ such that for all $x\in\Omega$, $D(\mu)(x)=D_{\rm uni}(\mu)(x):=\mu \mathbf {1}_{\Omega}(x)$ (mono-material plate, cf.~left panel of Figure~\ref{fig:plate}).
	
	\subsection{Synthetic Data Generation}\label{sec:synth}
	We generate the data by first synthesizing a true solution and then applying to it the linear functionals by means of their Riesz representations in the observable space $\mathcal U_M$.
	In order to synthesize the true solution, we consider a `true model' based on the bi-material plate (cf. right panel of Figure~\ref{fig:plate}) where we choose a fixed internal diffusivity $D_{\rm int}=1$ and define, for each $\mu\in\mathcal P$, the diffusivity function $D(\mu)$ as
	$D(\mu)(x)=D_{\rm syn}(\mu)(x):=\mu D_{\rm int}\mathbf {1}_{\Omega_{\rm ext}}(x)+ D_{\rm int}\mathbf 1_{\Omega_{int}}(x)$, for all $x\in\Omega$, where $\Omega_{\rm int} =(-1,1)^2$ and
	$\Omega_{\rm ext} =(-2,2)^2\setminus(-1,1)^2$,
	so that $\overline\Omega = \overline\Omega_{\rm int}\cup\overline\Omega_{\rm ext}$ and $\Omega_{\rm int}\cap\Omega_{\rm ext}=\emptyset$.
	The synthetic true solutions are then defined as the solutions of~\eqref{HTE}$-$\eqref{eq:nonlin_bcs}  for all $\mu\in\mathcal P$.

	\subsection{Test case (a) : Future forecast}\label{sec:test}
	
	For time discretization, we consider the time interval $I = [0,2.5]$s,
	a constant time step $\tau^k=1.25\times 10^{-2}$s for all $k \in \mathbb K^\mathrm{tr}$, and the set of discrete time nodes $\mathbb K^\mathrm{tr} = \{1,\ldots,200\}$.
	The parameter of the true model is fixed to $\mu=15$, both for training as well as testing.
	We first build the \texttt{bk} space. 
	Using a Proper Orthogonal decomposition (POD), we obtain $N = 4$ basis functions.
	We train our LSTM-RNN (cf.~Section~\ref{sec:svda}) with 2 additional dense hidden layers and 1 dense output layer using the data on the first $50$ time steps. 
	Using a mean squared error loss function and the Adam optimizer with a learning rate $10^{-2}$, we perform the SVDA prediction for the remaining future time steps. 
	We refer the reader to~\cite{hastie_elements} for more details on how to fit a neural network (number of hidden layers, optimizer choice, etc).
	We recall that the non-deterministic nature of the LSTM-RNN propagates to the SVDA. 
	Hence, the output results are tendencies of the solution behavior and not exactly reproducible solutions.
	For the values of the lookback parameter $lb\in\{1,7\}$ (cf. Section \ref{sec:svda}), Figure~\ref{fig:time_err} displays the relative $L^2$ errors to the true solution.
	
	We clearly see that the machine learning surrogates improve the accuracy of the future prediction in comparison to using only the \texttt{bk} model (in green). 
	Comparable results --- errors of about one percent --- are also achieved if we increase the training time window to the first $100$ time steps (data not shown).
	Note that the orange line, which depicts the (unrealistic) situation that true observations are available for data assimilation also at future time steps, is \emph{not} a lower bound for the SVDA error. The reason is that the data assimilation problem is constrained and thus the pure PBDW approach may not find the best solution in the discrete space. Consequently, the SVDA may find an even better solution candidate in the same space resulting in a smaller error value. 
	As expected, errors grow over time as we depart from the time window with true observations available at close previous steps, but the error growth is rather moderate overall.
	Since the results are quite similar for the different lookback values, we will focus on $lb=1$ in the following.
	
	\begin{figure}[htb]
		\includegraphics[width=0.48\textwidth,trim = 20mm 0mm 45mm 15mm, clip=true]{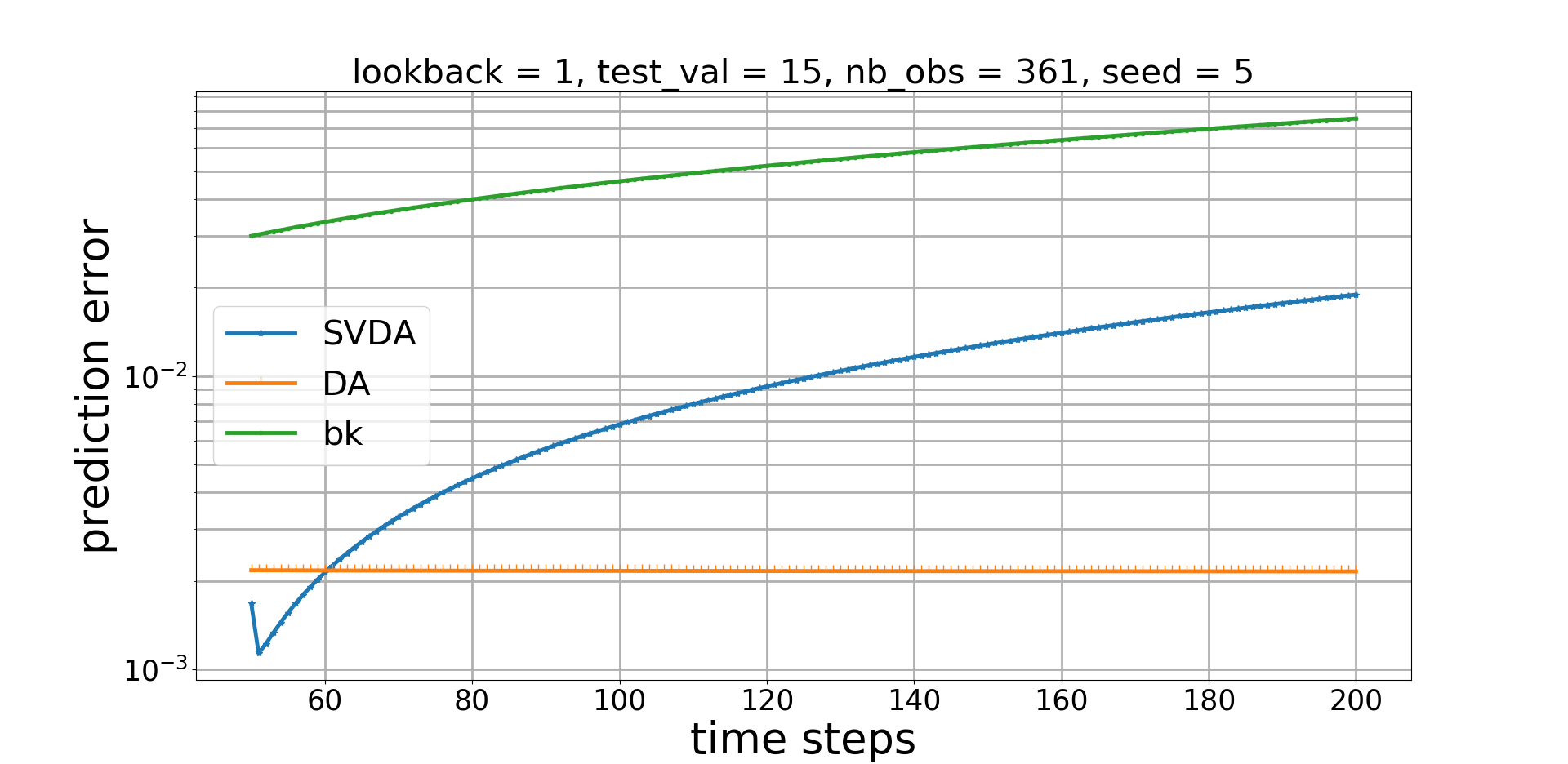}
		\includegraphics[width=0.48\textwidth, trim = 20mm 0mm 45mm 15mm, clip=true]{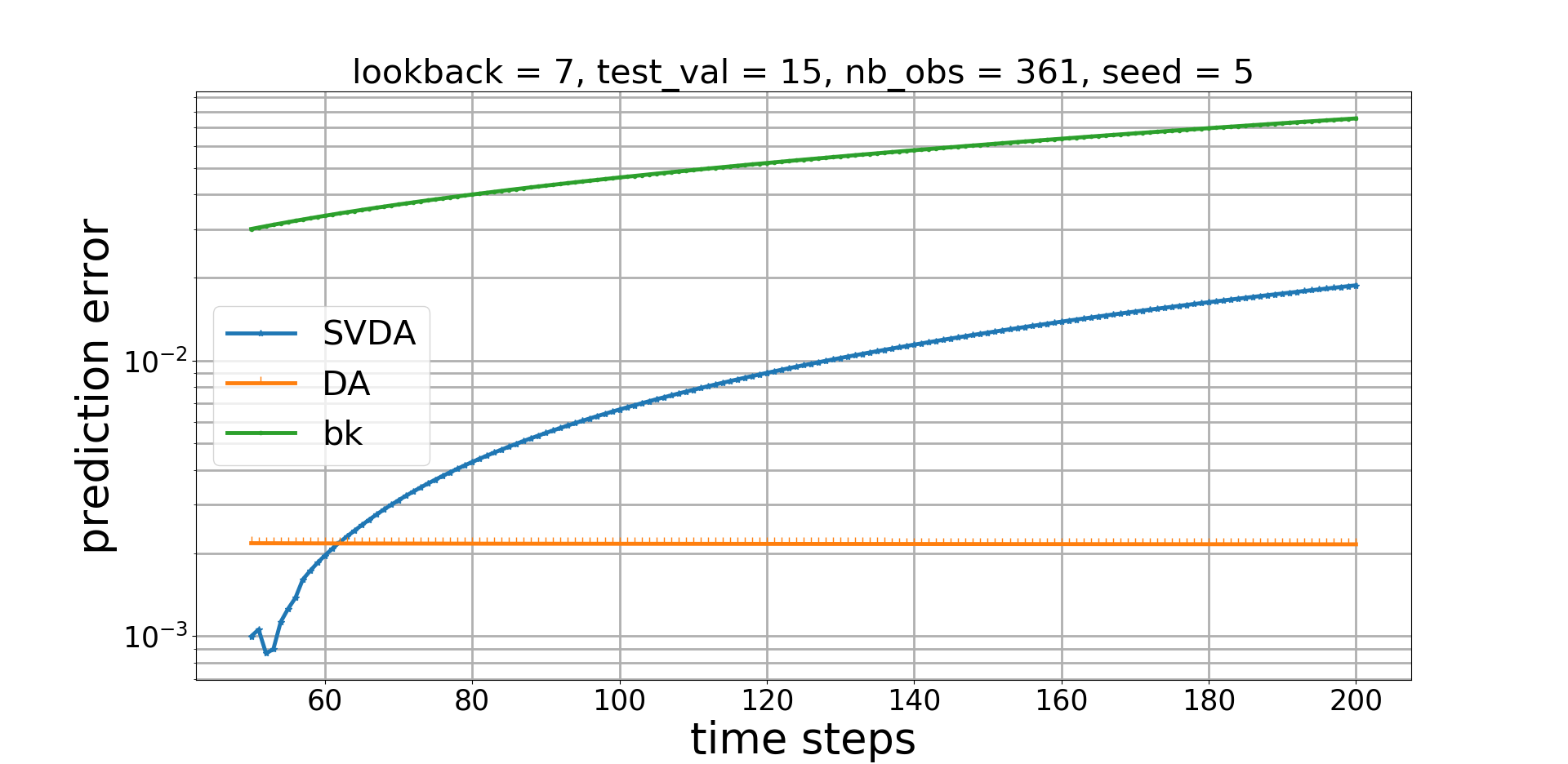} 
		\centering 
		\caption{Error estimation for $lb=1$ (left) and $lb=7$ (right).}
		\label{fig:time_err}
	\end{figure}

	\subsection{Test case (b) : Parametric forecast}\label{sec:testa}
	For a parametric forecast, we train the LSTM-RNN using the true model with $\mu=15$ over the whole time interval $[0,2.5]s$. The time steps, the spatial discretization and the network architecture are the same as in Test case (a).
	We use the trained LSTM-RNN to create surrogate data at all time steps for the value $\mu=17$ and integrate those surrogate observations into the SVDA.
	Figure~\ref{fig:parametric_lb1} shows the relative $L^2$ errors to the true solution.
	As announced, we focus on the case $lb=1$. 
	Although the parameter change from $\mu=15$ for the \texttt{bk} model to $\mu=17$ for the true model is very moderate, the green line for the \texttt{bk} error grows immediately  and the average \texttt{bk} error is about $4.5\cdot 10^{-2}$.
	Again, the SVDA is able to improve the accuracy  and its average error is about $6\cdot10^{-3}$. 
	Interestingly, the error even drops quite drastically somewhere in the middle of time simulation. 
	We believe that the nonlinearity of the model as well as the statistical training allow for such an error behavior, which is rather uncommon for traditional numerical methods.
	Finally, we mention that comparable results for the SVDA were even obtained when testing for other parameters -- even if they were rather different from the \texttt{bk} parameter $\mu=15$.
	
	\begin{figure}[htb]
		\centering 
		\includegraphics[width=0.68\textwidth, trim = 20mm 0mm 45mm 15mm, clip=true]{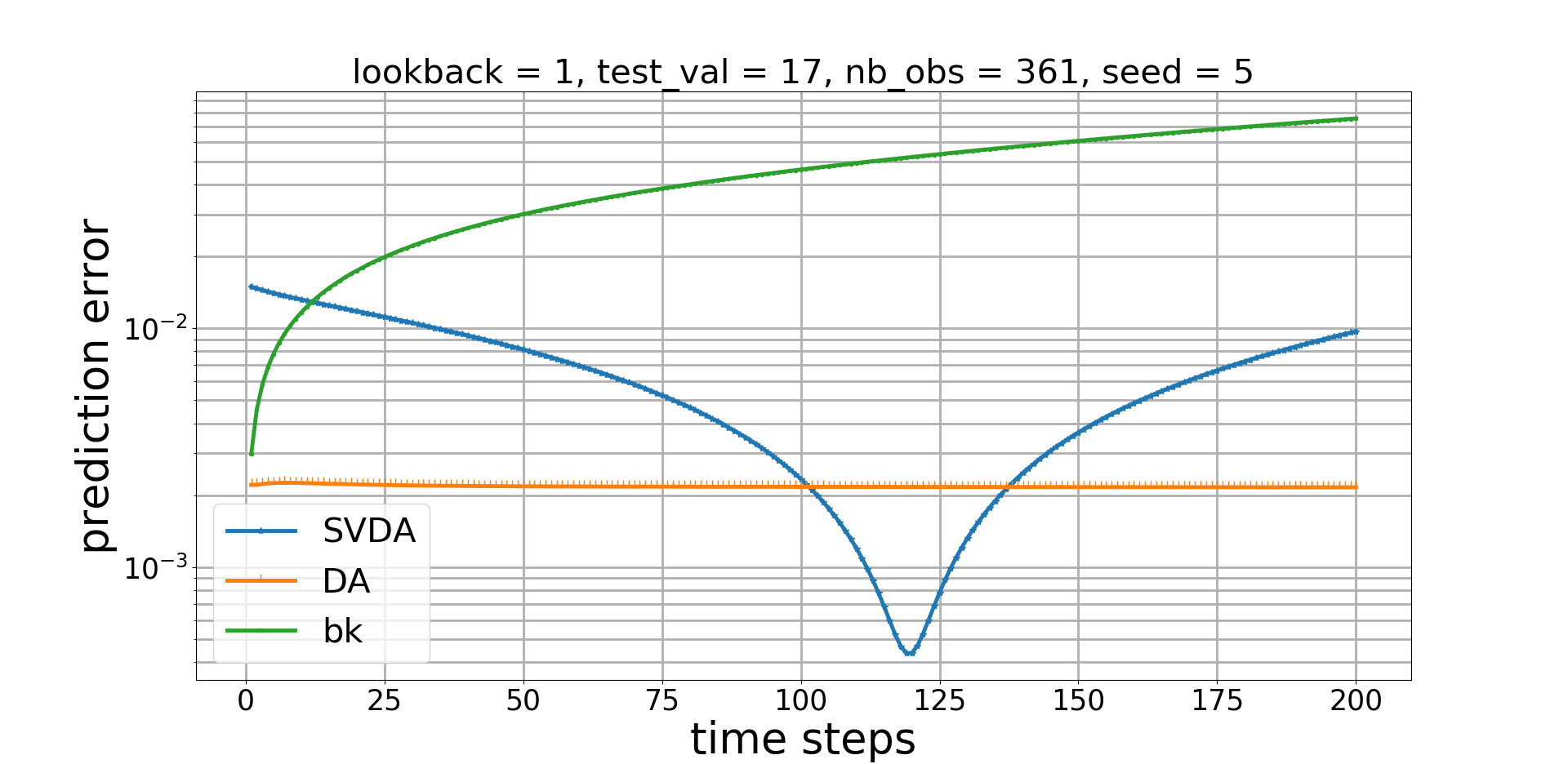} 
		\caption{Error estimation in the parametric case.}
		\label{fig:parametric_lb1}
	\end{figure}

	\section*{Conclusion}
	We introduced a statistical variational data assimilation (SVDA) method providing a new concept that combines machine learning methods with traditional data assimilation.
	While we explained the details of the framework based on long-short term memory (LSTM) networks and the parametric background data weak (PBDW) approach for data assimilation, the key idea is very flexible and versatile and can be easily adapted to other neural networks and data assimilation schemes.
	The core idea is to train a neural network based on available observations and use the network's predictions as surrogate data in situations where no observations are available, but data assimilation is needed or wished.
	We rigorously proved that the overall SVDA error is bounded above by the quality of discrete (approximation) spaces used in data assimilation and the quality of the machine learning model. 
	To illustrate the performance and applicability of the SVDA in practice, we considered a parametric heat equation with nonlinear boundary conditions. Both for future and parametric forecasts, the SVDA showed promising results. The flexibility of the framework entails that we cannot test the plethora of potential applications. Further investigations for more realistic scenarios (e.g. weather prediction) and on the influence of the various method parameters (lookback, neural network choice, etc.) are interesting future research directions.

	\section*{Acknowledgments}
	This work is funded by the Federal Ministry of Education and Research (BMBF) and the Baden-Württemberg Ministry of Science as part of the Excellence Strategy of the German Federal and State Governments.
	The authors acknowledges TEEMLEAP members Peter Knippertz, Sebastian Lerch, Uwe Ehret, Jörg Meyer, Julian Quinting, and Jannik Wilhelm for stimulating discussions.
	The help of the student assistant Klara Becker in running additional simulations is gratefully acknowledged.
	This work was conducted while both authors where affiliated with Karlsruhe Institute of Technology.
	
	\bibliographystyle{abbrv}
	\bibliography{teemleap}
\end{document}